\documentclass[a4paper,11pt,english]{amsart}

\usepackage{amsmath,amssymb,amsthm,graphicx,pstricks,xypic,comment,fancyhdr}
\usepackage[all]{xy}
\usepackage{epic}
\usepackage{eufrak}
\usepackage{lscape}
\usepackage[english]{babel}

\title{A derived equivalence between cluster equivalent algebras }
\author{Claire Amiot}
\address{Institut de Recherche Math\'ematique Avanc\'ee, 7 rue Descartes 67084 Strasbourg, France}
\email{amiot@math.unistra.fr}

\thanks{supported by the Storforsk-grant 167130 from the Norwegian Research Council}

\newcommand{\Hom}{{\sf Hom }}
\newcommand{\End}{{\sf End }}
\newcommand{\Ext}{{\sf Ext }}

\renewcommand{\mod}{{\sf mod \hspace{.02in}  }}

\newcommand{\Jac}{{\sf Jac}}

\newcommand{\ind}{{\sf ind \hspace{.02in} }}

\newcommand{\Sub}{{\sf Sub \hspace{.02in} }}
\newcommand{\fl}{{\sf f.l. \hspace{.02in} }}
\newcommand{\Fac}{{\sf Fac \hspace{.02in} }}
\newcommand{\add}{{\sf add \hspace{.02in} }}
\newcommand{\SSS}{\mathbb{S}_2}
\newcommand{\gldim}{{\sf gldim}}

\newcommand{\ten}{\otimes}
\newcommand{\lten}{\overset{\boldmath{L}}{\ten}}

\newcommand{\Cc}{\mathcal{C}}
\newcommand{\Dd}{\mathcal{D}}

\newcommand{\Ii}{\mathcal{I}}
\newcommand{\Jj}{\mathcal{J}}
\newcommand{\Tt}{\mathcal{T}}
\newcommand{\Mm}{\mathcal{M}}

\newcommand{\Ss}{\mathcal{S}}

\newcommand{\ww}{\mathbf{w}}
\newcommand{\bsm}{\begin{smallmatrix}}
\newcommand{\esm}{\end{smallmatrix}}

\newcommand{\leftsub}[2]{{\vphantom{#2}}_{#1}{#2}}

\numberwithin{equation}{section}

\newtheorem{thm}{Th{\'e}or{\`e}me}[section]
\newtheorem{thma}{Theorem}[section]
\newtheorem*{thm*}{Th{\'e}or{\`e}me}
\newtheorem{lema}[thma]{Lemma}
\newtheorem*{lem*}{Lemme}

\newtheorem*{prop*}{Proposition}
\newtheorem{prop}[thma]{Proposition}
\theoremstyle{remark}

\newtheorem{rema}[thm]{Remark}

\theoremstyle{definition}

\newtheorem{dfa}[thma]{Definition}
\newtheorem{construction}[thma]{Construction}

\begin{document}

\maketitle
\begin{abstract}
Let $Q$ be an acyclic quiver. Associated with any element $w$ of the Coxeter group of $Q$, triangulated categories $\underline{\Sub}\Lambda_w$ were introduced in \cite{Bua2}. For any reduced expression $\ww$ of $w$, the categories $\underline{\Sub}\Lambda_w$ are shown to be triangle equivalent to generalized cluster categories $\Cc_{\Gamma_\ww}$ associated to algebras $\Gamma_\ww$ of global dimension $\leq 2$ in \cite{ART}. For $\ww$ satisfying a certain property, called co-$c$-sortable, other algebras $A_\ww$ of global dimension $\leq 2$ are constructed in \cite{Ami3,AIRT} with a triangle equivalence $\Cc_{A_\ww}\simeq \underline{\Sub}\Lambda_w$. The main result of this paper is that the algebras $\Gamma_\ww$ and $A_\ww$ are derived equivalent when $\ww$ is co-$c$-sortable. The proof constructs explicitly a tilting module using the 2-APR-tilting theory introduced in \cite{IO}.
\end{abstract}

\section*{Introduction}
Let $k$ be an algebraically closed field. The cluster category $\Cc_Q$ associated to an acyclic quiver $Q$ has been introduced in \cite{Bua}. It is defined as the orbit category $\Dd^b(kQ)/\SSS$, where $\SSS$ is the composition of the Serre functor $\mathbb{S}$ of the bounded derived category $\Dd^b(kQ)$ of finitely presented $kQ$-modules with the second desuspension $[-2]$. This is a triangulated category \cite{Kel}, with finite dimensional spaces of morphisms ($\Hom$-finite for short), and with the 2-Calabi-Yau property: for any two objects $X$ and $Y$ in $\Cc_Q$ there is a functorial isomorphism
$\Hom_{\Cc_Q}(X,Y)\simeq D\Hom(Y,X[2])$ where $D$ is $\Hom_k(-,k)$. This construction was motivated, via \cite{Mar}, by the theory of cluster algebras initiated by Fomin and Zelevinsky  \cite{FZ1}. Following another point of view, Geiss, Leclerc and Schr\"oer have related in \cite{Gei} and \cite{Gei3} certain cluster algebras with the stable categories $\underline{\mod}\Lambda$ where $\Lambda$ is the preprojective algebra associated with a Dynkin quiver. These categories are also $\Hom$-finite, triangulated and 2-Calabi-Yau. Cluster categories $\Cc_Q$ and stable categories $\underline{\mod}\Lambda$ have both a special kind of objects  called \emph{cluster-tilting}. These are defined to be objects without selfextension and maximal with respect to this property. They are very important since they are the analogs of clusters.

\medskip
Therefore it is interesting to study $\Hom$-finite, 2-Calabi-Yau triangulated categories with cluster-tilting objects in general, and to find new such categories. To an acyclic quiver $Q$ and to an element $w$ of the Coxeter group of $Q$, Buan, Iyama, Reiten and Scott have associated in \cite{Bua2} (see also \cite{Gei2}, \cite{Gei4}) a triangulated category $\underline{\Sub}\Lambda_w$ where $\Sub \Lambda_w$ is a subcategory of the category of finite length modules over the preprojective algebra $\Lambda$ associated with $Q$. These categories are $\Hom$-finite, 2-Calabi-Yau and have cluster-tilting objects. Moreover they generalize the previous categories: If $Q$ is Dynkin and $w$ is the element of maximal length of the Coxeter group, $\underline{\Sub}\Lambda_w$ is equivalent to $\underline{\mod}\Lambda$. For any acyclic $Q$ (which is not $A_n$ with the linear orientation), if $w=cc$ where $c$ is the Coxeter element associated with the orientation of $Q$ then $\underline{\Sub}\Lambda_w$ is equivalent to the cluster category $\Cc_Q$.  

More recently cluster categories have been generalized in \cite{Ami3}, replacing the finite dimensional hereditary algebras $kQ$ by finite dimensional algebras $A$ of global dimension $\leq 2$. The orbit category $\Dd^b(A)/\SSS$ is not triangulated in general. Therefore the \emph{generalized cluster category} $\Cc_A$ is defined to be the triangulated hull of the orbit category $\Dd^b(A)/\SSS$. This construction generalizes again the previous constructions. Indeed in \cite{ART} given an acyclic quiver $Q$ and a reduced expression $\ww$ of an element $w$ in the Coxeter group, an algebra $\Gamma_\ww$ of global dimension $\leq 2$ is constructed with a triangle equivalence $$\Cc_{\Gamma_\ww}\simeq \underline{\Sub}\Lambda_w.$$  The algebra $\Gamma_\ww$ is constructed by using a natural grading on the preprojective algebra $\Lambda$.

With a very different point of view, it is shown in \cite{Ami3, AIRT} that for a certain kind of words called co-$c$-sortable, where $c$ is a Coxeter element, (containing the \emph{adaptable} words $\ww$ in the sense of \cite{Gei4}), it is possible to construct a triangle equivalence $$\Cc_{A_\ww}\simeq \underline{\Sub}\Lambda_w$$ where $A_\ww$ is the Auslander algebra of a finite torsion class in $\mod kQ$ naturally associated with the word $\ww$. 

The aim of this paper is to link the algebras $\Gamma_\ww$ and $A_\ww$ when $\ww$ is co-$c$-sortable. By definition of the generalized cluster category, if two algebras of global dimension $\leq 2$ are derived equivalent their generalized cluster categories are triangle equivalent. However the converse is not true. Two algebras of global dimension $\leq 2$ can have the same generalized cluster categories without being derived equivalent. We say in this case that they are \emph{cluster equivalent}. The main result of this paper is that the algebras $\Gamma_\ww$ and $A_\ww$ are derived equivalent (Theorem \ref{derivedeq}). Moreover we explicitly describe a tilting module yielding this equivalence using $2$-APR-tilting introduced in \cite{IO}.

\medskip
The paper is organized as follows. Section 1 is devoted to background definitions and results from \cite{Ami3}, \cite{Bua2} and \cite{BIRSm}.  We recall the definitions of the generalized cluster categories and of the categories $\Sub\Lambda_w$ and state some of their properties. In section 2 we recall results of \cite{ART} and \cite{AIRT}: We give the explicit construction of the algebra $\Gamma_\ww$ of global dimension $\leq 2$ and we describe the finite torsion class associated to a co-$c$-sortable word $\ww$. Sections 3 and 4 are devoted to the proof of the main theorem. We construct a tilting module over the algebra $A_\ww$ in section 3, and we prove that its endomorphism algebra is isomorphic to $\Gamma_\ww$ in section 4. In section 5 we give an example to illustrate the main theorem.

\subsection*{Acknowledgements}
The author would like to thank Idun Reiten for helpful comments and corrections on this paper, and the Research Council of Norway for financial support. She also would like to thank an anonymous referee for thoroughly reading the paper and useful comments.

\section{Background: Preprojective algebras and generalized cluster categories}

We assume all our algebras to be finite dimensional algebras over an algebraically closed field $k$. All modules are finite dimensional right modules unless otherwise stated, and the composition of arrows in a quiver is from right to left.

\subsection{Generalized cluster categories}

This section is devoted to recalling some results of \cite{Ami3}.

Let $\Gamma$ be a finite dimensional $k$-algebra of global dimension
$\leq 2$. We denote by $\Dd^b(\Gamma)$ the bounded derived category
of finite dimensional right $\Gamma$-modules. It has a Serre functor $-\lten_\Gamma D\Gamma$ that
we denote by $\mathbb{S}$. We denote by $\SSS$ the composition
$\mathbb{S}[-2]$.

The \emph{generalized cluster category} $\Cc_\Gamma$ of $\Gamma$ has been defined in \cite{Ami3} as
the triangulated hull of the orbit category $\Dd^b(\Gamma)/\SSS$. We
will denote by $\pi_\Gamma$ (or $\pi$ if there is no danger of confusion) the triangle functor

$$\xymatrix{\pi_\Gamma:\Dd^b(\Gamma)\ar@{->>}[r] &
  \Dd^b(\Gamma)/\SSS\ar@{^(->}[r] & \Cc_\Gamma}.$$
The case when $\End_{\Cc}(\pi \Gamma)$ is finite dimensional is especially nice since in this case the generalized cluster category contains special objects called \emph{cluster-tilting} (that is objects $T\in \Cc$ satisfying $\add(T)=\{X\in \Cc ,\ \Hom_\Cc(X,T[1])=0\}$ where $\add T$ is the additive closure of $T$). 
\begin{thma}[Theorem 4.10 of \cite{Ami3}]
Let $\Gamma$ be a finite dimensional  algebra of global dimension
$\leq 2$, and assume that the endomorphism algebra $\End_\Cc(\pi \Gamma)$ is finite dimensional.
Then $\Cc_\Gamma$ is a $\Hom$-finite,
triangulated $2$-Calabi-Yau category and $\pi(\Gamma)$ is a cluster-tilting object.
\end{thma}

The construction of $\Cc_\Gamma$ depends only on the derived category $\Dd^b(\Gamma)$. 

\begin{prop}\cite[Cor. 7.16]{AO}
 Let $\Gamma$ and $\Gamma'$ be two derived equivalent finite dimensional $k$-algebras of global 
dimension $\leq 2$ and assume that the endomorphism algebra $\End_{\Cc_\Gamma}(\pi \Gamma)$ is finite dimensional. Then the categories $\Cc_\Gamma$ and $\Cc_{\Gamma'}$ are triangle equivalent. 
\end{prop}

However, the converse is not true in general (see Example 5.7 in \cite{AO}). This leads us to state the following definition. 

\begin{dfa}
Two finite dimensional algebras $A$ and $B$ of global dimension $\leq 2$ are called \emph{cluster equivalent} if there exists a triangle equivalence between their generalized cluster categories $\Cc_A$ and $\Cc_B$. 
\end{dfa} 

\subsection{Categories associated to elements in the Coxeter group}

This section is devoted to recalling some results of \cite{Bua2} and \cite{BIRSm}.

Let $Q$ be a finite quiver without oriented cycles. We denote by $Q_0=\{1,\ldots, n\}$ the set of vertices and
 by $Q_1$ the set of arrows. The preprojective algebra associated to $Q$ is the algebra
$$k\overline{Q}/\langle \sum_{a\in Q_1} aa^*-a^*a\rangle$$
where $\overline{Q}$ is the double quiver of $Q$, which is obtained from $Q$ by adding to each arrow 
$a:i\rightarrow j\in Q_1$ an arrow $a^*:i\leftarrow j$ pointing in the opposite direction. 
We denote by $\Lambda$ the completion of the preprojective algebra associated to $Q$ and by $\fl\Lambda$ 
the category of right $\Lambda$-modules of finite length. 

For a vertex $i$ in $Q_0$ we denote by $\Ii_i$ the two-sided ideal $\Lambda(1-e_i)\Lambda$, where $e_i$ is the primitive idempotent of $\Lambda$ associated to the vertex $i$. Let $C_Q$ be the 
\emph{Coxeter group} associated to $Q$. It is defined by the generators $s_i$, where $i\in Q_0$, and by the relations
\begin{itemize}
 \item $s_i^2=1$,
\item $s_is_j=s_js_i$ if there is no arrow between $i$ and $j$,
\item $s_is_js_i=s_js_is_j$ if there is exactly one arrow between $i$ and $j$.
\end{itemize}

A \emph{reduced expression} $\ww=s_{u_1}\ldots s_{u_l}$ of an element $w$ of $C_Q$ is an expression of $w$ with $l$
 as small as possible. When $\ww=s_{u_1}\ldots s_{u_l}$ is  reduced, the integer $l$ is said to be the \emph{length} $l(w)$ of $w$.

Let $\ww=s_{u_1}\ldots s_{u_l}$ be a reduced expression of an element in $C_Q$. For $p\leq l$ we denote by 
$\Ii_{\ww_p}$ the two sided ideal $\Ii_{u_p}\Ii_{u_{p-1}}\ldots\Ii_{u_1}$. We denote by $\Lambda_w$ the algebra 
$\Lambda/\Ii_{\ww_l}$. As shown in \cite{Bua2}, the algebra $\Lambda_w$ depends only on the element $w$ and not on the choice of the reduced expression.  We denote by $\Sub\Lambda_w$ the category of submodules of finite dimensional free $\Lambda_w$-modules.

Let us recall Theorem III.2.8 of \cite{Bua2}.
 
\begin{thma}[Buan-Iyama-Reiten-Scott]\label{birsc}
 The category $\Sub\Lambda_w$ is a $\Hom$-finite Frobenius category and its stable category 
$\underline{\Sub}\Lambda_w$ is $2$-Calabi-Yau. For any reduced expression $\ww$ of $w$, the image $\underline{C}_\ww$ of the object $C_\ww=\bigoplus_{p=1}^le_{u_p}(\Lambda/\Ii_{\ww_p})\in\Sub \Lambda_w$ through the stabilisation $\Sub \Lambda_w\to \underline{\Sub}\Lambda_w$ is a cluster-tilting object in $\underline{\Sub}\Lambda_w$.
\end{thma}

The endomorphism algebra of $C_\ww$ is described in terms of a quiver with relations in \cite{BIRSm}.

Let us define the quiver $Q_\ww$ as follows:
\begin{itemize}
\item vertices: $1,\ldots,l(w)$.
\item for each $i\in Q_0$, one arrow $t\leftarrow s$ if $t$ and $s$ are two consecutive vertices of type $i$ 
(\emph{i.e.} $u_s=u_t=i$) and $t<s$ (we call them \emph{arrows going to the left});
\item for each $a:i\rightarrow j \in Q_1$, put $a:t\rightarrow s$ if $t$ is a vertex of type $i$, 
$s$ of type $j$, and if there is no vertex of type $i$ between $t$ and $s$ and if $s$ is the 
last vertex of type $j$ before the next vertex of type $i$ in the expression
  $\ww=s_{u_1}\ldots s_{u_l}$ (we call them \emph{$Q$-arrows})
\item for each $a:i\rightarrow j\in Q_1$, put $a^*: t\rightarrow
  s$ if $t$ is of type $j$, $s$ is of type $i$, if there is no vertex of type $j$ between $t$ 
and $s$
  and if $s$ is the last vertex of type $i$ before the next vertex of type $j$ in the expression $\ww=s_{u_1}\ldots s_{u_l}$ (we call them \emph{$Q^*$-arrows}).
\end{itemize}

For $i$ in $Q_0$ we define $l_i$ to be the maximal integer such that $u_{l_i}=i$. We 
denote by $Q'_\ww$ the full subquiver of $Q_\ww$ whose vertices are not $l_i$. 

For each $Q$-arrow $a:t\rightarrow s$ in $Q'_\ww$
 we denote by $W_{a}$ the
composition $aa^*p$ if there is a (unique) $Q^*$-arrow
$a^*:r\rightarrow t$ in $Q'_\ww$ where $u_r=u_s$ and where $p$ is the
composition of arrows going to the left $r\leftarrow \cdots \leftarrow s$. Otherwise we put
$W_a=0$. For
each $Q^*$-arrow $a^*:t\rightarrow r$ in $Q'_\ww$, we denote by $W_{a^*}$ the
composition $a^*a p$ if there exists a (unique) $Q$-arrow
$a:s\rightarrow t$ in $Q'_\ww$ with $u_s=u_r$ and where $p$ is the composition of arrows going to the left
$s\leftarrow \cdots \leftarrow r$. Otherwise we put
$W_{a^*}=0$. Then let $W_\ww$ be the sum
$$W_\ww=\sum_{a\ Q\textrm{-arrow}}W_{a}-\sum_{a^*\ Q^*\textrm{-arrow}}W_{a^*}.$$  

It is a \emph{potential} in the sense of \cite{DWZ}, that is, a linear combination of cycles in $Q'_\ww$. For a cycle $p$ in $Q'_\ww$ and an arrow $a$ in $Q'_{\ww}$, it is possible to define the partial derivative $\partial_ap$ as the sum $\partial_ap:=\sum_{p=uav}vu$. The definition of the partial derivative can be extended by linearity to any potential.

The \emph{Jacobian algebra} (see \cite{DWZ}) is defined as the algebra
$$\Jac(Q'_\ww, W_\ww):=kQ'_\ww/\langle \partial_aW_\ww, a\in (Q'_{\ww})_1\rangle.$$

Let us recall Theorem 6.6 of \cite{BIRSm}.

\begin{thma}[Buan-Iyama-Reiten-Smith]\label{birs}
Let $\ww=s_{u_1}\ldots s_{u_l}$ be a reduced expression of an element $w$ of the Coxeter group $C_Q$. Let $\underline{C}_\ww\in \underline{\Sub}\Lambda_w$ be the cluster-tilting object (defined in Theorem \ref{birsc}) associated to this reduced expression. Then there is an algebra isomorphism $$\Jac(Q'_\ww,W_\ww)\simeq \End_{\underline{\Sub}\Lambda_w}(\underline{C}_\ww).$$ 
\end{thma}

\section{Categories associated to a word as generalized cluster categories}

In this section we recall some results of \cite{ART} and \cite{AIRT} which describe some categories $\underline{\Sub}\Lambda_w$ as generalized cluster categories.

\subsection{General words: Results of \cite{ART}}

For any reduced expression $\ww$ of any element $w$ of the Coxeter group, the authors construct in \cite{ART} an algebra $\Gamma_\ww$ of global dimension 2 and an equivalence of triangulated categories $\underline{\Sub}\Lambda_w\simeq \Cc_{\Gamma_\ww}$. We recall here the construction of $\Gamma_\ww$.

Let $Q$ and $\Lambda$ be as in the previous section.
Let $\ww=s_{u_1}\ldots s_{u_l}$ be a reduced expression of an element $w$ in the Coxeter group $C_Q$. 
Since the category $\Sub\Lambda_w$ and the object $C_\ww$ do not depend on the orientation of $Q$, 
we can assume that the orientation of $Q$ satisfies the property
\[\tag{$*$}\quad \textrm{if there exists }i\rightarrow j, \textrm{ then }l_i<l_j,\]
where $l_i$ is the maximal integer such that $u_{l_i}=i$.

We define a grading on the quiver $Q'_\ww$. All arrows going to the left and all $Q$-arrows are defined to have degree 0. All $Q^*$-arrows are defined to have degree 1.
 It is then easy to see that the potential $W_\ww$ 
is homogeneous of degree 1. Hence we get a grading on the Jacobian algebra $\Jac(Q'_\ww,W_\ww)$, and therefore on the algebra  
$\End_{\underline{\Sub}\Lambda_w}(\underline{C}_\ww)$ by Theorem \ref{birs}. We denote by $\Gamma_\ww$ its part of degree zero.

\begin{thma}[Theorem 4.4 of \cite{ART}]\label{art}
 For any acyclic quiver $Q$ and any element $w$ in the Coxeter group of $Q$, the algebra $\Gamma=\Gamma_\ww$ is of global dimension $\leq 2$ and there exists a commutative diagram of triangle functors:
$$\xymatrix{\Dd^b(\Gamma)\ar[rr]^{F_\Gamma}\ar[d]^{\pi_\Gamma} && \underline{\Sub}\Lambda_w \\ 
\Cc_\Gamma \ar[urr]^{f_{\Gamma}}&&},$$
where $\Cc_\Gamma$ is the generalized cluster category associated to $\Gamma$. The functor $f_\Gamma$ is an equivalence, and we have $F_\Gamma(\Gamma)=\underline{C}_\ww$.
\end{thma}

\subsection{Co-$c$-sortable words}
This subsection is devoted to recalling some results of \cite{AIRT}.

Let $Q$ be a quiver without oriented cycles with $n$ vertices. We assume that the orientation of $Q$ satisfies 
\[\tag{$**$} \textrm{if there exists }i\rightarrow j, \textrm{ then }i<j.\]
We denote by $c$ the Coxeter element $s_1\ldots s_n$.

\begin{dfa}
 An element $w$ of the Coxeter group of $Q$ is called \emph{co-c-sortable} if there exists a reduced expression $\ww$ of $w$ of the form $\ww=c^{(m)}\ldots c^{(1)}c^{(0)}$, where all $c^{(t)}$ are subwords of $c$ whose supports satisfy 
$$supp(c^{(m)})\subseteq supp(c^{(m-1)})\subseteq \ldots \subseteq supp(c^{(1)})\subseteq supp(c^{(0)})\subseteq Q_0.$$
If $i\in Q_0$ is in the support of $c^{(t)}$, by abuse of notation, we will write $i\in c^{(t)}$.
\end{dfa}
\begin{rema} \label{remark}
\begin{enumerate}
\item The word $\ww$ is co-$c$-sortable if and only if $\ww^{-1}$ is $c^{-1}$-sortable in the sense of \cite{Rea}. 
\item The co-$c$-sortable expression $\ww$ is unique for a co-$c$-sortable element $w$ (cf \cite{Rea}). 
\item If $\ww$ is co-$c$-sortable, the conditions $(*)$ and $(**)$ for the orientation of $Q$ are the same.
\end{enumerate}
\end{rema}

Let $\ww=c^{(m)}\ldots c^{(1)}c^{(0)}$ be a co-$c$-sortable word. Let $Q^{(1)}$ be the full subquiver of $Q$ whose 
support is the same as $c^{(1)}$. Then the word $\ww'=c^{(m)}\ldots c^{(1)}$ is co-$c^{(1)}$-sortable as an element of the Coxeter group $C_{Q^{(1)}}$.
\begin{construction}\label{constructionT}
 For $t\geq 1$ and $i$ in $c^{(t)}$, we define $kQ^{(1)}$-modules $T_{(i,t)}$ by induction as follows: 
\begin{itemize}
 \item We put $T_{(i,1)}=e_iD(kQ^{(1)})$ for all $i\in c^{(1)}$.
\item For $t\geq 2$, assume that we have defined $T_{(j,s)}$ for $1\leq s\leq t-1$ and $j\in c^{(s)} $, and 
$T_{(n,t)},\ldots, T_{(i+1,t)}$.   Then $T_{(i,t)}$ is defined to be the kernel of the map 
$$f\colon E\rightarrow T_{(i,t-1)}$$  where $f$ is a minimal right
$\add(\bigoplus_{j<i}T_{(j,t-1)}\oplus\bigoplus_{j>i}T_{(j,t)}) $-approximation. 
\end{itemize}
For $i$ in $Q^{(1)}_0$, we define $m_i$ such that $i$ is in $c^{(m_i)}$ but not in $c^{(m_i+1)}$.
We define the $kQ^{(1)}$-module $$T:=\bigoplus_{i\in Q^{(1)}_0} T_{(i,m_i)}.$$
\end{construction}
Here are some results shown in \cite{AIRT}.
\begin{thma}\cite[Thm 3.20]{AIRT}\label{airt1} Let $\ww'$, $Q^{(1)}$ and $T$ be as above. Then the following holds:
\begin{itemize}
\item[(a)] the modules $T_{(i,t)}$ are indecomposable and pairwise non-isomorphic;
\item[(b)] $T$ is a tilting $kQ^{(1)}$-module with finite torsion class;
\item[(c)]  the torsion class $\Fac T=\{ X\in \mod kQ \textrm{ s.t. }
\Ext^1_{kQ^{(1)}}(T,X)=0\}$ is the additive category $\add\{T_{(i,t)},t\geq 1, i\in c^{(t)}\}$;
\item[(d)] the sequences $\xymatrix{0\ar[r]& T_{(i,t)}\ar[r] & E\ar[r]^(.4)f &
    T_{(i,t-1)}\ar[r] & 0}$ are exact and are the almost split sequences
  of $\Fac(T)$.
\end{itemize}

\end{thma} 
We have also the following result which is a generalization of Theorem 5.21 of \cite{Ami3}.
\begin{thma}\cite[Thm 3.23]{AIRT}\label{airt2}
 Let $Q$, $\ww=c^{(m)}\ldots c^{(0)}$, $\ww'=c^{(m)}\ldots c^{(1)}$, $Q^{(1)}$ and  $T_{(i,t)}$ (for $t\geq 1$ and $i\in c^{(t)}$) be as above. Define the endomorphism algebra 
$$A:=\End_{kQ^{(1)}}(\bigoplus_{t\geq 1}\bigoplus_{i\in c^{(t)}} T_{(i,t)}).$$ Then the algebra $A$ is of global dimension $\leq 2$ 
and there exists a commutative diagram of triangle functors:
$$\xymatrix{\Dd^b(A)\ar[rr]^{F_A}\ar[d]^{\pi_A} && \underline{\Sub}\Lambda_w \\ 
\Cc_A \ar[urr]^{f_{A}}&&},$$
where $\Cc_A$ is the generalized cluster category associated with the algebra $A$. The functor $f_A$ is an equivalence, and we have $F_A(A)=\underline{C}_\ww$.
\end{thma}

\subsection{Main result}
From now on we assume that $\ww=c^{(m)}\ldots c^{(1)}c^{(0)}$ is a co-$c$-sortable word. Combining Theorems \ref{art} and \ref{airt2}, 
we get two algebras $A$ and $\Gamma$ of global dimension $\leq 2$, with the following diagram

$$\xymatrix{\Dd^b(A)\ar@{..>}[rr]^{?}\ar[d]_{\pi_A}\ar[dr]^{F_A} &  & \Dd^b(\Gamma)\ar[d]^{\pi_\Gamma}\ar[dl]_{F_\Gamma}\\ 
 \Cc_A\ar[r]^(.4)\sim& \underline{\Sub}\Lambda_w &\Cc_\Gamma\ar[l]_(.4)\sim}.$$

As we have seen in the first section, we do not automatically get a derived equivalence between $A$ and $\Gamma$.
The aim of this paper is to prove that there is an equivalence in this case and that 
this derived equivalence is given by a tilting module which is easy to describe.
More precisely we will show the following in the next two sections.

\begin{thma}\label{derivedeq}
Let $\ww=c^{(m)}\ldots c^{(0)}$ be a co-$c$-sortable element in the Coxeter group of $Q$. For $t\geq 1$ and $i$ in $c^{(t)}$, define $kQ^{(1)}$-modules $T_{(i,t)}$ as in Construction~\ref{constructionT}.  Let $A=\End_{kQ^{(1)}}(\bigoplus_{t\geq 1}\bigoplus_{i\in c^{(t)}} T_{(i,t)})$ be the algebra as in Theorem~\ref{airt2}. And let $\Gamma$ be the degree zero part of the graded algebra $\Jac(Q'_\ww,W_\ww)$ as defined in subsection 2.1.

For $p\geq 1$ and $j\in c^{(p)}$ define the indecomposable projective $A$-module 
$$P_{(j,p)}:=\Hom_{kQ^{(1)}}(\bigoplus_{t\geq 1}\bigoplus_{i\in c^{(t)}} T_{(i,t)}, T_{(j,p)}),$$ 
and the complex
$$M=\bigoplus_{p=1}^m\bigoplus_{j\in
  c^{(p)}}\SSS^{-p+1}(P_{(j,p)})\in \Dd^b(A),$$
where $\SSS$ is the autoequivalence $\mathbb{S}[-2]$ of $\Dd^b(A)$.
Then the following holds:
\begin{enumerate}
\item $M$ is a tilting module;
\item $\End_A(M)\simeq\Gamma$;
\item  the functor $R\Hom_A(M,-)$ makes the following diagram commute
$$\xymatrix{\Dd^b(A)\ar[rr]^{R\Hom_A(M,-)}\ar[dr]_{F_A} &  & \Dd^b(\Gamma)\ar[dl]^{F_\Gamma}\\ 
 & \underline{\Sub}\Lambda_w &}.$$
\end{enumerate}
\end{thma}

\begin{rema} Note that we have equivalences $$ \gldim A=1 \  \Leftrightarrow \  m=1 \ \Leftrightarrow \ \gldim\Gamma =1 \Leftrightarrow \ A\simeq \Gamma.$$ Therefore we can assume from now that $m\geq 2$. \end{rema}

\section{Construction of a tilting module}

This section is devoted to the proof of Theorem \ref{derivedeq}(1). 
 We start in subsection 3.1 with preliminaries on module categories over Auslander algebras. Then we state some general lemmas on tilting modules over hereditary algebras in subsection 3.2. In subsection 3.3 we describe explicitely the repeated action of $\SSS$ on indecomposable projective $A$-modules. Finally we recall results on 2-APR-tilting of \cite{IO} in subsection 3.4 and prove Theorem \ref{derivedeq}(1) in subsection 3.5. 

\subsection{Module categories over Auslander algebras}
Let $\Mm$ be an additive $k$-category with finite dimensional $\Hom$-spaces and with finitely many indecomposables up to isomorphism. We denote by $\ind \Mm$ a set of  representative of each isomorphism class of indecomposables in $\Mm$. Let  $A$ be the \emph{Auslander algebra} of $\Mm$, that is, the endomorphism algebra  $A=\End(\bigoplus_{X\in \ind \Mm}X)$. This is a finite-dimensional basic algebra. Denote by $\mod A$ the category of finite dimensional right $A$-modules and by $\mod \Mm$ the category of finitely presented functors $\Mm^{op}\rightarrow \mod k$.
Then the functor 
$$\Hom( \bigoplus_{X\in \ind Mm}X,-): \mod \Mm\rightarrow \mod A$$ is an equivalence of category.  Through this equivalence,  indecomposable projective $A$-modules are isomorphic to the functors 
of the form $\Mm(-,X)$ where $X$ is an indecomposable object in $\Mm$ and indecomposable injective $A$-modules
are isomorphic to the functors of the form $D\Mm(X,-)$ where $X$ is an indecomposable object in $\Mm$.

For $\Tt$ a full subcategory of $\Mm$ and $X$ an object of $\Mm$, we define the $\Mm$-module $\Mm(-,X)/[\Tt]$ as the cokernel of 
$$\Mm(-,T)\rightarrow \Mm(-,X)$$ induced by a minimal right $\Tt$-approximation $T\rightarrow X$.

We first state a lemma which describes morphisms between objects in $\mod \Mm$ of the form $\Mm(-,X)/[\Tt]$ 
in terms of morphisms in $\Mm$.

\begin{lema}\label{space}
Let $\Mm$ be an additive $k$-category with finitely many
indecomposables. Let $\Tt$
and $\Ss$ be full subcategories of $\Mm$, then we have an isomorphism
for any $X$ and $Y$ in $\Mm$ between
$$\Hom_{\mod \Mm}(\Mm(-,X)/[\Tt],\Mm(-,Y)/[\Ss])$$ and the space of commutative
squares up to homotopy $$\xymatrix{T\ar[r]^t\ar[d] & X\ar[d]
  \\S\ar[r]^s & Y}$$ where $t$ (resp. $s$) is a minimal right $\Tt$(resp. $\Ss$)
-approximation in $\Mm$.
\end{lema}

\begin{proof}
Let $t:T\rightarrow X$ be a minimal $\Tt$- approximation of $X$. Then the
projective presentation of the module $\Mm(-,X)/[\Tt]$ is 
$$\xymatrix{\Mm(-,T)\ar[rr]^{\Mm(-,t)} && \Mm(-,X)\ar[r] & \Mm(-,X)/[\Tt]\ar[r] &
  0}$$
Thus the space of morphisms $\Hom_A(\Mm(-,X)/[\Tt], \Mm(-,Y)/[\Ss])$
is isomorphic to the space of commutative squares 
 $$\xymatrix{\Mm(-,T)\ar[rr]^{\Mm(-,t)}\ar[d] & &\Mm(-,X)\ar[d]
  \\\Mm(-,S)\ar[rr]^{\Mm(-,s)} && \Mm(-,Y)}$$
up to homotopy, where $s:S\rightarrow Y$ is a minimal right $\Ss$-approximation. By the Yoneda lemma, this is isomorphic to the space
of commutative squares
 $$\xymatrix{T\ar[r]^t\ar[d] & X\ar[d]
  \\S\ar[r]^s & Y}$$ up to homotopy.

\end{proof}

\subsection{The category $\Fac(T)$}

In the rest of the section, we assume that $\ww=c^{(m)}\ldots c^{(0)}$ is a co-$c$-sortable word with $m\geq 2$. For $t\geq 1$ and $i\in c^{(t)}$, we define $kQ^{(1)}$-modules $T_{(i,t)}$ and $T:=\bigoplus_{i\in c^{(1)}}T_{(i,m_i)}$, where $m_i$ is the maximal integer such that $i\in c^{(m_i)}$ as in Construction~\ref{constructionT}. By Theorem~\ref{airt1}, it is a tilting module and we have  
$$\Fac(T)=\{X\in \mod kQ^{(1)},\ \Ext^1_{kQ^{(1)}}(T,X)=0\}=\add\{T_{(i,t)}, 1\leq t\leq m, i\in c^{(t)}\}.$$ 
Define $A:= \End_{kQ^{(1)}}(\bigoplus_{t\geq 1, i\in c^{(t)}} T_{(i,t)})$ and $\Mm=\Fac(T)$. Since the $T_{(i,t)}$ are indecomposable and pairwise non isomorphic (Theorem \ref{airt1} (a)), $A$ is the Auslander algebra of $\Mm$.

The category $\Mm$, as a torsion class, has almost split sequences (cf Theorem \ref{airt1}(d)). 
We will denote by $\tau$ its Auslander-Reiten translation (which is a functor by \cite[Section 3]{AS}). By Theorem \ref{airt1}(d), for any $1\leq t\leq m-1$ and $i\in c^{(t)}$ we have $\tau T_{(i,t)}=T_{(i,t+1)}$ if $i\in c^{(t+1)}$ and $0$ else. Therefore by Construction~\ref{constructionT} we have $T_{(i,t)}\simeq \tau^{t-1} (e_i D kQ^{(1)})$ and since $\Mm=\add\{T_{(i,t)}, 1\leq t\leq m, i\in c^{(t)}\}$ any indecomposable object in $\Mm$ is in the $\tau$-orbit of a direct factor of $DkQ^{(1)}$.

The following lemma is classical from tilting theory, we include here the proof for the convenience of the reader.

\begin{lema}\label{lemme approx}
Let $X$ be an indecomposable object in $\Fac (T)$. Then there exists a short exact sequence $\xymatrix{0\ar[r]& T_1\ar[r]& T_0\ar[r]^f & X\ar[r] & 0}$ where $f$ is a minimal right $\add(T)$-approximation and $T_1,T_0\in \add(T)$.
\end{lema}

\begin{proof}
Let $f:T_0\rightarrow X$ be a minimal right $\add(T)$-approximation. It is surjective since $\Fac(T)$ coincide with the set of modules generated by $T$. Then form the exact sequence \begin{equation}\label{suite exacte}\xymatrix{0\ar[r]& T_1\ar[r]& T_0\ar[r]^f & X\ar[r] & 0}.\end{equation} It induces a long exact sequence
$$\xymatrix@-.5cm{0\ar[r] & \Hom_{kQ^{(1)}}(-,T_1)\ar[r] & \Hom_{kQ^{(1)}}(-,T_0)\ar[r]^{f^*} &\Hom_{kQ^{(1)}}(-,X)\ar[r] & \Ext^1_{kQ^{(1)}}(-,T_1)\ar[r] & \Ext^1_{kQ^{(1)}}(-,T_0) }.$$
Since $f$ is an $\add T$ approximation the map $f^*_{|_{\add T}}$ is surjective. Moreover $\Ext^1_{kQ^{(1)}}(T,T_0)$ vanishes. Hence $\Ext^1_{kQ^{(1)}}(T,T_1)$ vanishes.

From the sequence \ref{suite exacte} we also obtain 
$$\xymatrix{0=\Ext^1_{kQ^{(1)}}(T_0,T)\ar[r] & \Ext^1_{kQ^{(1)}}(T_1,T)\ar[r] & \Ext^2_{kQ^{(1)}}(X,T)=0}.$$
Since $T$ is a tilting $kQ^{(1)}$-module, from  $\Ext^1_{kQ^{(1)}}(T,T_1)=\Ext^1_{kQ^{(1)}}(T_1,T)=0$, we deduce $T_1\in \add T$.

\end{proof}

\begin{lema}\label{ARformula}
Let $X,Y\in \Mm$ such that $\tau^{-1} X\neq 0$. Then we have a functorial isomorphism
$$D\Ext^1_{kQ^{(1)}}(\tau^{-1}X,Y)\simeq \Mm(Y,X).$$
\end{lema}

\begin{proof}
The category $\Mm$ is functorially finite and extension closed. It has an Auslander-Reiten formula by \cite{EMM}, that is for any $X$ and $Y$ in $\Mm$ we have a functorial isomorphism
 $$D\Ext^1_{kQ^{(1)}}(\tau^{-1}X,Y)\simeq \Mm(Y,X)/[\add DkQ^{(1)}].$$
 A morphism $\Mm(Y,X)$ factors through an object in $\add DkQ^{(1)}$ if and only if $X\in \add DkQ^{(1)}$ since $DkQ^{(1)}$ is a slice. But if $\tau^{-1}X\neq 0$ then $X$ is not in $\add DkQ^{(1)}$, thus we get the result.
\end{proof}
\subsection{Action of $\SSS$ on projective $A$-modules}

For $t\geq 1$ we define the subcategory $\Tt_t$ of $\Mm$ as
$$\Tt_t:=\add( T\oplus\tau^{-1}
T\oplus \cdots \oplus \tau^{-t+1}T),$$ and by convention $\Tt_0=\emptyset$. Note that  $T_{(i,t)}=\tau^{t-m_i} T_{(i,m_i)}\in \Tt_{m_i-t+1}$, thus we have $\Tt_{m+1}=\Mm$.

The following lemma  is a variant of Lemma 5.9 of
\cite{Ami3} and will be very useful in the rest of the paper.

\begin{lema}\label{Phiproj}

Let $X$ be an object of $\Mm$ and $n\geq 0$ such that $\tau^{-n}X\neq 0$, where $\tau$ is the AR-translation of the category $\Mm$. Then we have an isomorphism in $\mod \Mm$:
$$\SSS^{-n}(\Mm(-,X))\simeq \Mm(-,\tau^{-n}X)/[\Tt_{n}]$$
\end{lema}

\begin{proof}
We prove this lemma by induction on $n$.

Let $X$ be an indecomposable in $\Mm$ with $\tau^{-1}X\neq 0$. Since $\tau^{-1}X$ is in $\Mm$ we have a short exact
sequence by Lemma~\ref{lemme approx}
\begin{equation}\label{approx}
\xymatrix{0\ar[r] & T_1\ar[r] & T_0\ar[r]^f& \tau^{-1}X\ar[r] &
0}\end{equation}
with $T_0$ and $T_1$ in $\Tt_1=\add(T)$ and $f$ a minimal right $\add(T)$-approximation. The objects $T_0$ and $T_1$ are not zero since $X$ is not zero and in $\Mm$.
Thus we get an exact sequence in $\mod \Mm$
$$\xymatrix{0\ar[r] & \Mm(-,T_1)\ar[r] & \Mm(-,T_0)\ar[r] &
  \Mm(-,\tau^{-1}X)\ar[r] & \Mm(-,\tau^{-1}X)/[\Tt_1]\ar[r] & 0}$$
which gives a projective resolution of the module
$\Mm(-,\tau^{-1}X)/[\Tt_1]$.
Therefore the object $$\SSS( \Mm(-,\tau^{-1}X)/[\Tt_1])$$ is isomorphic in $\Dd^b(\mod \Mm)$ to the
complex
$$\xymatrix@-.2cm{ D\Mm(T_1,-)\ar[r] & D\Mm(T_0,-)\ar[r] &
  D\Mm(\tau^{-1}X,-)},$$ where $D\Mm(T_1,-)$ is in degree 0.
From the short exact sequence (\ref{approx}) we get a long exact sequence in
$\mod \Mm$:
$$\xymatrix@-.5cm{ D\Ext^1_{kQ^{(1)}}(T_0,-)_{|_{\Mm}}\ar[r] &
  D\Ext^1_{kQ^{(1)}}(\tau^{-1}X,-)_{|_\Mm}\ar[r]  &D\Mm(T_1,-)\ar[r] & D\Mm(T_0,-)\ar[r] &
  D\Mm(\tau^{-1}X,-)\ar[r] & 0}.$$
Since $\Mm=\Fac(T)$, we have $\Ext^1_{kQ^{(1)}}(T_0,\Mm)=0$.
By Lemma~\ref{ARformula} we have an isomorphism in $\mod \Mm$
$$D\Ext^1_{\Mm}(\tau^{-1}X,-)\simeq\Mm(-,X).$$ 
Hence we get the desired isomorphism
$$\SSS^{-1}(\Mm(-,X))\simeq\Mm(-,\tau^{-1}X)/[\Tt_1],$$ which is the assertion for $n=1$.

Now let $n\geq 2$ and assume that for any $Y$ with $\tau^{-n+1}Y\neq 0$ we
have $$\SSS^{-n+1}(\Mm(-,Y))\simeq \Mm(-,\tau^{-n+1}Y)/[\Tt_{n-1}].$$
Let $X$ be in $\Mm$ such that $\tau^{-n}X$ is not zero. Then by the assertion for $n=1$, we have an isomorphism
\begin{equation}\label{ison=1}\SSS^{-n}(\Mm(-,X))\simeq\SSS^{-n+1}(\Mm(-,\tau^{-1}X)/[\Tt_1]).\end{equation}

By the above short exact sequence (\ref{approx}), we obtain that $\Mm(-,\tau^{-1}X)/[\Tt_1]$ is isomorphic in $\Dd^b(\mod \Mm)$ to the complex 
$$\xymatrix{  \Mm(-,T_1)\ar[r] & \Mm(-,T_0)\ar[r] & \Mm(-,\tau^{-1}X) }.$$ 
Applying the induction hypothesis  for $Y=T_0, T_1,\tau^{-1}X$ we obtain the isomorphisms
\begin{align*}\SSS^{-n+1}(\Mm(-,T_1)) &\simeq \Mm(-,\tau^{-n+1}T_1)/[\Tt_{n-1}]\\ \SSS^{-n+1}(\Mm(-,T_0)) &\simeq \Mm(-,\tau^{-n+1}T_0)/[\Tt_{n-1}]\\ \SSS^{-n+1}(\Mm(-,X)) &\simeq \Mm(-,\tau^{-n+1}X)/[\Tt_{n-1}].
\end{align*}
Hence $\SSS^{-n+1}(\Mm(-,\tau^{-1}X)/[\Tt_1])$ is isomorphic in $\Dd^b(\mod \Mm)$ to the complex 
 \begin{equation}\label{iso2}
\xymatrix{\Mm(-,\tau^{-n+1}T_1)/[\Tt_{n-1}]\ar[r]
    &\Mm(-,\tau^{-n+1}T_0)/[\Tt_{n-1}]\ar[r]
      &\Mm(-,\tau^{-n}X)/[\Tt_{n-1}] }\end{equation}

Since $\tau^{-n}X$ is not zero, 
the short exact sequence (\ref{approx})
yields  a short exact
sequence $$\xymatrix{0\ar[r] & \tau^{-n+1}T_1\ar[r] &
  \tau^{-n+1}T_0\ar[r] & \tau^{-n}X\ar[r] & 0}.$$
  The objects $\tau^{-n+1}T_0$
and $\tau^{-n+1}T_1$ cannot be zero since $X,\tau^{-1}X,\ldots \tau^{-n}X$ are not zero.
As above, we obtain an exact sequence  in $\mod \Mm$ 
$$\xymatrix{ 0\ar[r]& \Mm(-,\tau^{-n+1}T_1)\ar[r] & \Mm(-,\tau^{-n+1}T_0)\ar[r] & \Mm(-,\tau^{-n}X)\ar[r] & \Mm(-,\tau^{-n}X)/[\add \tau^{-n+1}T]\ar[r] & 0}.$$
Dividing by the ideal $[\Tt_{n-1}]$ we obtain an exact sequence
$$\xymatrix@-.4cm{ 0\ar[r]& \Mm(-,\tau^{-n+1}T_1)/[\Tt_{n-1}]\ar[r] & \Mm(-,\tau^{-n+1}T_0)/[\Tt_{n-1}]\ar[r] & \Mm(-,\tau^{-n}X)/[\Tt_{n-1}]\ar[r] &  \Mm(-,\tau^{-n}X)/[\Tt_n]\ar[r] & 0},$$
that is $\Mm(-,\tau^{-n}X)/[\Tt_n]$ is isomorphic in $\Dd^b(\mod \Mm)$ to the complex
\begin{equation}\label{iso3}\xymatrix{  \Mm(-,\tau^{-n+1}T_1)/[\Tt_{n-1}]\ar[r] & \Mm(-,\tau^{-n+1}T_0)/[\Tt_{n-1}]\ar[r] & \Mm(-,\tau^{-n}X)/[\Tt_{n-1}]}.\end{equation}
Combining the isomorphisms~\eqref{ison=1},\eqref{iso2} and \eqref{iso3}, we obtain the isomorphisms $$\Mm(-,\tau^{-n}X)/[\Tt_{n}]\simeq \SSS^{-n+1}(\Mm(-,\tau^{-1}X)/[\Tt_1])\simeq \SSS^{-n}(\Mm(-,X)).$$
This finishes the induction.
\end{proof}

\subsection{2-APR-tilting}
The object $M$ of Theorem \ref{derivedeq} is constructed by applying powers of the functor $\SSS$ to summands of $A$. To prove that it is tilting, we use the tool of 2-APR-tilting introduced by Iyama and Oppermann. The following result is  Theorem~4.5 together with Proposition~4.7 of \cite{IO}.

\begin{thma}[Iyama-Oppermann]\label{OsamuSteffen}
Let $B$ be a finite dimensional $k$-algebra of global
dimension at most $2$. Suppose we can decompose $B=P\oplus Q$ as a $B$-module
in such a way that
\begin{enumerate}
\item $\Hom_B(Q,P)=0$;
\item $\Ext^{-1}_{\Dd^b(B)}(Q,\SSS^{-1}(P))=0$.
\end{enumerate}
Then $T=\SSS^{-1}(P)\oplus Q$ is a tilting complex over $B$ and
$\End_B(T)$ is of global dimension $\leq 2$.
\end{thma}

Applying recursively this theorem, we will prove the following.

\begin{prop}\label{OsamuSteffen2}
Let $B$ be  a finite dimensional $k$-algebra of global dimension
$\leq 2$. Suppose that we can decompose $B$ as the sum
$P_m\oplus\cdots \oplus P_1\oplus P_0$ of $B$-modules in such a
way that 
\begin{itemize}
\item[(a)]
for any $s,t,j$ such that $s-j-t\geq 1$ and $j\geq 0$
the space $\Hom_{\Dd^b(B)}(P_t,\SSS^{-j} P_s)$ vanishes;
\item[(b)] $\SSS^{-j}P_s$ is a module for $0\leq j\leq s$.
\end{itemize}
Then $T=\SSS^{-m}P_m\oplus\ldots \SSS^{-1}P_1\oplus P_0$ is a tilting
module  and the algebra
$\End_B(T)$ is of global dimension at most $2$.
\end{prop}
 
\begin{proof}
We prove by induction on $i\geq 0$ that the object $$T_i=\SSS^{-i}(P_m\oplus \cdots \oplus P_i)\oplus
\SSS^{-i+1}P_{i-1}\oplus \cdots \oplus \SSS^{-1}P_1 \oplus P_0$$
is a tilting module over $B$ and that the endomorphism algebra
$B_i:=\End_{B}(T_i)$ has
global dimension $\leq 2$. This holds for $i=0$ by hypothesis. Suppose that this holds for an $i\geq 0$.
The functor $F_i=R\Hom_B(T_i,-)$ yields a
triangle equivalence 
$$\xymatrix{F_i:\Dd^b(B)\ar@(dl,dr)_{\SSS=-\lten_BDB[-2]}\ar[rr]^\sim
  &&\Dd^b(B_i)\ar@(dl,dr)_{\leftsub{i}\SSS=-\lten_{B_i}DB_i[-2]}}$$ which sends $T_i$ to $B_i$. By the uniqueness of the
Serre functor we have an isomorphism 
$$F_i\circ \SSS=R\Hom_B(T_i,-\lten_B DB[-2])\simeq
R\Hom_B(T_i,-)\lten_{B_i}DB_i[-2]=\leftsub{i}\SSS\circ F_i.$$ We want to apply
Theorem \ref{OsamuSteffen} to $$P:=F_i(\SSS^{-i}(P_m\oplus\cdots \oplus P_{i+1}))\textrm{ and } Q:=F_i(\SSS^{-i}(P_i)\oplus\cdots\oplus \SSS^{-1}(P_1)\oplus P_0).$$ 
We have $$\begin{array}{rcl}
\Hom_{\Dd^b(B_i)}(Q,P) & = & \Hom_{\Dd^b(B_i)}( F_i(\SSS^{-i}(P_i)\oplus\cdots\oplus \SSS^{-1}(P_1)\oplus P_0),F_i(\SSS^{-i}(P_m\oplus\cdots\oplus P_{i+1})))\\ & \simeq &\Hom_{\Dd^b(B)}(\SSS^{-i}(P_i)\oplus\cdots\oplus \SSS^{-1}(P_1)\oplus P_0, \SSS^{-i}(P_m\oplus\cdots \oplus P_{i+1}))\\ & = & 0\textrm{ by (1).}\end{array}$$
Moreover, we have
$$\begin{array}{rcl} \Ext^{-1}_{\Dd^b(B_i)}(Q,\leftsub{i}\SSS^{-1}P) & = & \Ext^{-1}_{\Dd^b(B_i)}( F_i(\SSS^{-i}(P_i)\oplus\cdots\oplus \SSS^{-1}(P_1)\oplus P_0),\leftsub{i}\SSS^{-1}F_i(\SSS^{-i}(P_m\oplus\cdots \oplus P_{i+1})))\\ & \simeq & \Ext^{-1}_{\Dd^b(B_i)}(F_i(\SSS^{-i}(P_i)\oplus\cdots\oplus \SSS^{-1}(P_1)\oplus P_0),F_i(\SSS^{-(i+1)}(P_m\oplus\cdots \oplus P_{i+1})))\\ 
& \simeq & \Ext^{-1}_{\Dd^b(B)}(\SSS^{-i}(P_i)\oplus\cdots\oplus \SSS^{-1}(P_1)\oplus P_0,\SSS^{-(i+1)}(P_m\oplus\cdots \oplus P_{i+1}))\end{array}$$
By (b), for $1\leq j\leq s$ the object $\SSS^{-j}P_s$ is a module, hence the space $\Ext^{-1}_B(P_l, \SSS^{-j}P_s)$ vanishes for any $l$. Therefore the space $\Ext^{-1}_{\Dd^b(B_i)}(Q,\leftsub{i}\SSS^{-1}P)$ vanishes. 

Thus by Theorem \ref{OsamuSteffen}, $\leftsub{i}\SSS^{-1}(P)\oplus Q\simeq F_i(T_{i+1})$ is a tilting complex in $\Dd^b(B_i)$. Therefore $T_{i+1}$ is a tilting complex in $\Dd^b(B)$. It is a module by $(2)$, and its endomorphism algebra $B_{i+1}=\End_B(T_{i+1})\simeq \End_{B_i}(F_i(T_{i+1}))$ is of global dimension $\leq 2$. Thus we get the proposition.
\end{proof}

\subsection{Application to our setup}

In this subsection, we apply Proposition~\ref{OsamuSteffen2} for $B=A=\bigoplus_{t=1}^m (\bigoplus_{i\in c^{(t)}}P_{(i,t)})$ where $P_{(i,t)}:=\Mm(-,T_{(i,t)})$ is the projective indecomposable $A$-module defined in Theorem \ref{derivedeq}.

\begin{prop}\label{tilting}
The complex
$M=\bigoplus_{t=1}^m\bigoplus_{i\in
  c^{(t)}}\SSS^{-t+1}(P_{(i,t)})$
is a tilting $A$-module.
\end{prop}

\begin{proof}
If $i$ is in $c^{(t)}$, then $\tau^{-t+1}(T_{(i,t)})$ is
isomorphic to $e_iD(kQ^{(1)})$.
Now for $t=1,\ldots,m$, we denote by $P_t$ the projective
$A$-module $\bigoplus_{i\in c^{(t)}}P_{(i,t)}$.
For $i$ in
$c^{(t)}$, the indecomposable projective $P_{(i,t)}$ is of the form
$\Mm(-,T_{(i,t)})$. Thus we have
$P_{(i,t)}=\Mm(-,\tau^{t-1}(e_i DkQ^{(1)}))$. Therefore we can write
$$P_t=\Mm(-,\tau^{t-1}(DkQ^{(1)})).$$ Note that if $i$ is not in $c^{(t)}$
then $\tau^{t-1}(e_iDkQ^{(1)})$ is zero, thus the decomposition of $P_t$
into indecomposables is given by $$P_t=\bigoplus_{i\in
  c^{(t)}}\Mm(-,\tau^{t-1}(e_iDkQ^{(1)}))=\bigoplus_{i\in c^{(t)}}\Mm(-,T_{(i,t)}).$$

We want to apply Proposition \ref{OsamuSteffen2} to the decomposition
$A=P_m\oplus\cdots \oplus P_1$. By Lemma \ref{Phiproj}, we know
that for any $0\leq j \leq s-1$, we have $$\SSS^{-j}P_s\simeq \Mm(-,
\tau^{s-1-j}DkQ^{(1)})/[\Tt_{j}]$$ which is a module. Thus we have condition
(b) of Proposition \ref{OsamuSteffen2}.

For $s-j-t\geq 1$ we have isomorphisms
$$\begin{array}{rcll}
\Hom_{A}(P_t,\SSS^{-j}P_s) & = &
\Hom_{A}(\Mm(-,\tau^{t-1}(DkQ^{(1)})),\SSS^{-j}(\Mm(-,\tau^{s-1}(DkQ^{(1)}))))& \\
& = &
\Hom_{A}(\Mm(-,\tau^{t-1}(DkQ^{(1)})),\Mm(-,\tau^{s-1-j}(DkQ^{(1)}))/[\Tt_{j}]) & \textrm{ by Lemma \ref{Phiproj}}\\
 & \simeq & \Mm(\tau^{t-1}(DkQ^{(1)}),\tau^{s-j-1}(DkQ^{(1)}))/[\Tt_{j}] &\textrm{ by Lemma \ref{space}} \\ 
& \simeq & \Mm(DkQ^{(1)},
\tau^{s-j-t}(DkQ^{(1)}))/\add(\tau^{1-t}T\oplus \cdots\oplus
\tau^{2-t-j}(T))&\end{array}$$
Since $s-j-t\geq 1$ the space
$\Mm(DkQ^{(1)},\tau^{s-j-t}DkQ^{(1)})$ vanishes. Hence we have condition (a) of
Proposition \ref{OsamuSteffen2}. Therefore the complex
$$M=\bigoplus_{t=1}^m\SSS^{-t+1}(P_t)=\bigoplus_{t=1}^m\bigoplus_{i\in
  c^{(t)}}\SSS^{-t+1}(P_{(i,t)})$$ is a tilting module. 

\end{proof}

\section{Computation of the endomorphism algebra}

In this section, we prove Theorem  \ref{derivedeq}(2), that is that the endomorphism algebra $\End_A(M)$ is isomorphic to the algebra $\Gamma$ defined in section 2.1.
The strategy consists of describing these two algebras with a quiver and an ideal of relations.

Let $\ww=c^{(m)}\ldots c^{(0)}$ be a co-$c$-sortable word, and define $\ww'=c^{(m)}\ldots c^{(1)}$.  
Let $R_{\ww'}$ be the following quiver:
\begin{itemize}
\item its vertices are $(i,t)$ where $i$ is in $c^{(t)}$;
\item for $i\in Q_0^{(1)}$, for $t\geq 1$ such that $i$ is in
  $c^{(t+1)}$, we have an arrow $q^i_t:(i,t)\rightarrow (i,t+1)$;
\item for any $a:i\rightarrow j$ in $Q_1^{(1)}$ such that $i,j\in c^{(t)}$ we have an arrow
  $a_t:(i,t)\rightarrow (j,t)$.
\end{itemize}

We define an ideal $\Jj_{\ww'}$ of relations on the path algebra $kR_{\ww'}$
generated by commutative squares
$$\xymatrix{(i,t)\ar[r]^{q^i_t}\ar[d]^{a_t} &
  (i,t+1)\ar[d]^{a_{t+1}}\\ (j,t)\ar[r]^{q^j_t} & (j,t+1)}$$ when all
these arrows are defined, and by zero relations
$$\xymatrix{(i,t)\ar[r]^{a^i_t} & (j,t)\ar[r]^{q^j_t} & (j,t+1)}$$
when $i$ is not in $c^{(t+1)}$.

\begin{lema}\label{iso1}
The algebra $\Gamma$ is isomorphic to the algebra
$kR_{\ww'}/\Jj_{\ww'}$.
\end{lema}

\begin{proof}
In the case where $\ww$ is co-$c$-sortable, the quivers $Q_\ww$ and $Q'_\ww$ described in section 1 are much simpler.  The orientation of $Q$ satisfies condition $(*)$ of section 2.1 if and only if it satisfies $(**)$ of section 2.2 (cf Remark \ref{remark}(3)).
It is routine to check that if we remove the $Q^*$-arrows of $Q'_\ww$ we get $R_{\ww'}$, and that 
the partial derivatives $\partial_{a^*}W_\ww$ where $a^*$ is a $Q^*$-arrow are exactly the relations generating $\Jj_{\ww'}$. 
\end{proof}

\begin{prop}\label{iso}
There exists an algebra isomorphism $$\bar{G}:kR_{\ww'}/\Jj_{\ww'}\rightarrow
\End_{A}(\bigoplus_{t=1}^m\bigoplus_{i\in
  c^{(t)}}\SSS^{-t+1}(P_{(i,t)}))=\End_A(M).$$
\end{prop}
\begin{proof}
We divide the proof in several steps.

\medskip
\noindent
\textit{Step 1: Construction of $\bar{G}$.}
\smallskip

We first define $G:kR_{\ww'}\to \End_A(M)$ on the vertices of $R_{\ww'}$. For $i$ in $c^{(t)}$ we put 
$$ G(i,t)=\SSS^{-t+1}(P_{(i,t)})=\SSS^{-t+1}(\Mm(-,\tau^{t-1}(e_i DkQ^{(1)})).$$

Let $s,t$ be integers $\leq m$, $i\in c^{(t)}$ and $j\in c^{(s)}$. By Lemma \ref{Phiproj}, we have an isomorphism
$$\Hom_{A}(\SSS^{-t+1}(P_{(i,t)}),\SSS^{-s+1}(P_{(j,s)})\simeq \Hom_{A}(\Mm(-,e_i(DkQ^{(1)}))/[\Tt_{t-1}],\Mm(-,e_j(DkQ^{(1)}))/[\Tt_{s-1}]).$$  By Lemma \ref{space}, the above $\Hom$-space  
is isomorphic to the space of commutative squares
$$\xymatrix{T_{t-1}\ar[r]\ar[d] & e_i(DkQ^{(1)})\ar[d] \\ T_{s-1}\ar[r] &
  e_j(DkQ^{(1)})}$$ up to homotopy, where horizontal maps are minimal right $\Tt_{t-1}$
(resp. $\Tt_{s-1}$)-approximations.

Hence to define a morphism $G:kR_{\ww'}\rightarrow \End_A(M)$ we have to construct  for any arrow $q^i_t:(i,t)\rightarrow(i,t+1)$ a commutative square $$U^i_t:=\xymatrix{T_{t-1}\ar[r]^-f\ar[d] & e_i(DkQ^{(1)})\ar@{=}[d] \\ T_{t}\ar[r]^-g &
  e_i(DkQ^{(1)})},$$
 and for any arrow $a_t:(i,t)\rightarrow (j,t)$ a commutative square
 $$S(a)_t:=\xymatrix{T_{t-1}\ar[r]^-f\ar[d] & e_i(DkQ^{(1)})\ar[d]^a \\ T'_{t-1}\ar[r]^-g &
  e_j(DkQ^{(1)})}.$$
   
Here is an immediate result which will often be used in the proof.

\begin{lema}\label{useful}
For $t\geq 1$ we have equivalences
$$i\notin c^{(t+1)} \ \Leftrightarrow \ \tau^t(e_i DkQ^{(1)})=0\ \Leftrightarrow\  e_i DkQ^{(1)}\in \Tt_t.$$
\end{lema}

Let $i$ be in $Q_0^{(1)}$ and $t$ be an integer $\geq 1$. Let $f:T_{t-1}\rightarrow e_i(DkQ^{(1)})$ be a minimal
right $\Tt_{t-1}$-approximation and let $g:T_{t}\rightarrow e_i (DkQ^{(1)})$ be a
minimal right $\Tt_{t}$-approximation. Since we have the inclusion
$\Tt_{t-1}\subset\Tt_{t}$, then there exists a commutative square 

$$U^i_t:=\xymatrix{T_{t-1}\ar[r]^-f\ar[d] & e_i(DkQ^{(1)})\ar@{=}[d] \\ T_{t}\ar[r]^-g &
  e_i(DkQ^{(1)})}.$$

It is homotopic to zero if and only if $e_i(DkQ^{(1)})$ is in
$\Tt_{t}$. This is equivalent to the fact that $i$ is not in $c^{(t+1)}$ by Lemma \ref{useful}.
Thus for $i$ in $c^{(t+1)}$ we define $$G(q_t^i)=U^i_t\neq 0.$$

Let $a:i\rightarrow j$ be an arrow in $Q_1^{(1)}$ and $t$ be an integer
$\geq 1$. Let $f:T_{t-1}\rightarrow e_i(DkQ^{(1)})$ and $g:T'_{t-1}\rightarrow
e_j(DkQ^{(1)})$ be minimal right $\Tt_{t-1}$-approximations. Then we have a
commutative square 
$$S(a)_t:=\xymatrix{T_{t-1}\ar[r]^-f\ar[d] & e_i(DkQ^{(1)})\ar[d]^a \\ T'_{t-1}\ar[r]^-g &
  e_j(DkQ^{(1)})}$$

If this square is homotopic to zero then there exists a map
$h:e_i(DkQ^{(1)})\rightarrow T'_{t-1}$ such that $a=g\circ h$. Since $a$ is an
irreducible map, $h$ is a section or $g$ is a retraction. Thus
$e_i(DkQ^{(1)})$ or $e_j(DkQ^{(1)})$ are in $\Tt_{t-1}$. By Lemma \ref{useful}, this means that either $i\notin c^{(t)}$ or $j\notin c^{(t)}$. 

Therefore for any $a:i\rightarrow j $ in $Q_1$ and for $t\geq 1$ such that $i,j\in c^{(t)}$, we put $$G(a_t)=S(a)_t\neq 0.$$

Now it remains to check that the map $G:kR_{\ww'}\rightarrow \End_A(M)$ vanishes on $\Jj_{\ww'}$.
For any $i,j$ in $c^{(t)}$ we have a commutative diagram:
$$\xymatrix@R=.3pc{T_{t-1}\ar[rr]\ar[dd]\ar[dr] && e_i(DkQ^{(1)})\ar[dr]\ar@{=}[dd] & \\
& T'_{t-1}\ar[rr]\ar[dd] && e_j(DkQ^{(1)})\ar@{=}[dd]\\
T_{t}\ar[rr]\ar[dr] && e_i(DkQ^{(1)})\ar[dr] & \\ & T'_{t}\ar[rr] && e_j
(DkQ^{(1)})}$$
This implies that in $\End_{A}(M)$ we have the relation
$U^j_t\circ S(a)_{t} =S(a)_{t+1}\circ U^i_t$ if $S(a)_{t+1}$ is not zero,
that is if $i$ and $j$ are in $c^{(t+1)}$. Moreover we have the relation
$U^j_t\circ S(a)_{t}=0$ if $i$ is not in $c^{(t+1)}$.

Therefore the morphism $G:kR_{\ww'}\rightarrow \End_A(M)$ factors through morphism
$\bar{G}:kR_{\ww'}/\Jj_{\ww'}\rightarrow \End_A(M).$
 
\bigskip
\noindent
\textit{Step 2: The map $\bar{G}$ is surjective.}

\smallskip
We will show that the squares of the form $S(a)_t$ and $U^i_t$ generate the algebra $\End_A(M)$.

Let $\alpha$ be a path in $Q^{(1)}$ from $i$ to $j$. We denote by $length(\alpha)$ its length. For a commutative square 
$$S:=\xymatrix{T_{t-1}\ar[r]\ar[d] & e_i(DkQ^{(1)})\ar[d]^\alpha \\ T_{s-1}\ar[r] &
  e_j(DkQ^{(1)})}$$
 let us define the size of $S$ by $$size(S)=s-t+length(\alpha).$$

For all integers $t\geq 1$, all $i$ in $Q^{(1)}_0$ and all $a$ in $Q^{(1)}_1$, we have
$size(U_t^i)=size(S(a)_t)=1.$
\smallskip

We first show that the only non zero squares $S$ with $size(S)\leq 0$ are the isomorphism and then $size(S)=0$.

Let $s\leq t$ be two integers, $i$ be in $c^{(t)}$ and $j$ be in
$c^{(s)}$. Suppose there is a commutative square 

$$S:=\xymatrix{T_{t-1}\ar[r]^-{f_t}\ar[d]^u & e_i(DkQ^{(1)})\ar[d]^\alpha \\ T_{s-1}\ar[r]^-{g_s} &
  e_j(DkQ^{(1)})}$$
where $\alpha$ is non zero path, and where $f_t$ (resp. $g_s$) is a minimal right $\Tt_{t-1}$ (resp. $\Tt_{s-1}$)-approximation. 
Since $s\leq t$, we have $\Tt_{s-1}\subset\Tt_{t-1}$. The approximation $f_t$ is not zero, hence $u$ is not zero and $T_{t-1}$ must be
in $\Tt_{s-1}$. Let $f_s:T'_{s-1}\rightarrow e_i(DkQ^{(1)})$ be a minimal
right $\Tt_{s-1}$-approximation, then we have such a factorization:
$$\xymatrix{T_{t-1}\ar[r]^-{f_t}\ar@/^/[d] & e_i(DkQ^{(1)})\ar@{=}[d] \\ T'_{s-1}\ar[r]^-{f_s}\ar@/^/[u] &
  e_i(DkQ^{(1)})}$$
The fact that the maps $f_t$ and $f_s$ gives an isomorphism between $\Mm(-,e_i(DkQ^{(1)}))/[\Tt_{s-1}]$ and
$\Mm(-,e_i(DkQ^{(1)}))/[\Tt_{t-1}]$, thus we have $s=t$.

Finally we get that all squares of size $<0$ are zero. Moreover, all
squares of size $0$ are isomorphisms and all squares of size 1 which are not homotopic to zero
are the $S(a)_t$ and $U^i_t$.
\smallskip

Now we will show that any square $S$ such that $size(S)\geq 2$ can be written as a composition of squares of size strictly smaller.
Let $s\geq t$ be positive integers, $\alpha\neq 0$ be a path from $i$ to $j$ in $Q^{(1)}$. Let $S$ be a non zero
commutative square non homotopic to zero with $size (S)\geq 2$:
$$\xymatrix{T_{t-1}\ar[r]^-{f_t}\ar[d]^u & e_i(DkQ^{(1)})\ar[d]^\alpha \\ T_{s-1}\ar[r]^-{g_s} &
  e_j(DkQ^{(1)})}$$
where $f_t$ (resp. $g_s$) is a minimal right $\Tt_{t-1}$ (resp. $\Tt_{s-1}$)-approximation.
Assume that $s\geq t+1$. Then we have a commutative diagram

$$\xymatrix@R=.5pc{&T_{t-1}\ar[dl]\ar[rr]^-{f_t}\ar[dd] & &e_i(DkQ^{(1)})\ar[dd]^\alpha\ar@{=}[dl] 
  \\ T_{t}\ar[dr]\ar[rr]^(.6){f_{t+1}} &&e_i(DkQ^{(1)})\ar[dr]&\\& T_{s-1}\ar[rr]^-{g_s} &&
  e_j(DkQ^{(1)})}$$
where $f_{t+1}$ is a minimal right $\Tt_t$ approximation.
Thus the square $S$ is the composition $B\circ U_t^i$ where $$B:=\xymatrix{ T_{t}\ar[r]^-{f_{t+1}}\ar[d] & e_i(DkQ^{(1)})\ar[d]^\alpha \\ T_{s-1}\ar[r]^-{g_s} &
  e_j(DkQ^{(1)})}.$$
We have
$size(B)=s-t-1+length(\alpha)$ and $size(U^i_t)=1$.

If $s=t$ and if $\alpha$ is a composition of arrows
$a_1\circ\cdots\circ a_n$ of $Q_1^{(1)}$ with $n\geq 2$, then we have $S=S(a_1)_t\circ B$
where $size(B)=length(\alpha)-1$. Therefore $(U^i_t,S(a)_t)$ generate the
algebra $\End(\bigoplus_{t}\bigoplus_{i\in
  c^{(t)}}\SSS^{-t+1}(P_{(i,t)})$ and the morphism $\bar{G}$ is surjective.

\bigskip
\noindent
\textit{Step 3: The map $\bar{G}$ is injective.}

\smallskip
Let $x$ be a linear combination of paths from $(i,t)$ to $(j,s)$ in $R_{\ww'}$ which is non zero in $kR_{\ww'}/\Jj_{\ww'}$. Then we have $s\geq t$, $i\in c^{(t)}$ and $j\in c^{(s)}$. The element $x$ can be written as a sum $\sum_ux_u$ where for each $u$ there is a path
$$\xymatrix{\alpha_u:=i=i_1\ar[r]^(.7){a^1} & i_2\ar[r]^{a^2} & \cdots \ar[r]^{a^{n-1}} & i_n\ar[r]^{a^n} & j}$$ in $Q^{(1)}$ such that $i_l\in c^{(s)}$ for $l=1,\ldots, n$ and
$$x_u= \lambda_u q^i_t q^i_{t+1} \ldots q^i_{s-1}a^n_sa^{n-1}_s\ldots a^1_s $$ where $\lambda_u$ is in the field $k$.

Now assume that $G(x)$
is a
commutative square homotopic to zero    
$$\xymatrix{T_{t-1}\ar[r]\ar[d] & e_i(DkQ^{(1)})\ar@{..>}[dl]^h\ar[d]^\alpha \\ T_{s-1}\ar[r] &
  e_j(DkQ^{(1)})}$$
where $\alpha=\sum_{u}\lambda_u \alpha_u$ and
where horizontal maps are minimal right $\Tt_{t-1}$ and $\Tt_{s-1}$-approximations.
Since $s\geq t$, we have a facorization:

$$\xymatrix@-.5pc{&T_{t-1}\ar[dl]\ar[rr]\ar[dd] & &e_i(DkQ^{(1)})\ar@{..>}@/_2pc/[ddll]_h\ar[dd]^\alpha\ar@{=}[dl] 
  \\ T'_{s-1}\ar[dr]\ar[rr] &&e_i(DkQ^{(1)})\ar[dr]&\\& T_{s-1}\ar[rr] &&
  e_j(DkQ^{(1)})}$$
Thus the square  
$$\xymatrix{T'_{s-1}\ar[r]\ar[d] & e_i(DkQ^{(1)})\ar@{..>}[dl]^h\ar[d]^\alpha \\ T_{s-1}\ar[r] &
  e_j(DkQ^{(1)})}$$ is homotopic to zero.

Therefore for all $u$ there exists a factorization
$$\xymatrix{\alpha_u:e_i(DkQ^{(1)})\ar[r]^{\beta_u} &
  e_{i_{l(u)}}(DkQ^{(1)})\ar[r]^{\gamma_u} & e_j(DkQ^{(1)})}$$ with $e_{i_{l(u)}}$ in
$\Tt_{s-1}$. Thus by Lemma \ref{useful} $i_{l(u)}$ is not in $c^{(s)}$ for all $u$. This is a contradiction. Therefore the morphism $\bar{G}$ is injective.

\end{proof}

\begin{proof}[Proof of Theorem~\ref{derivedeq}.] 
Combining Proposition \ref{tilting} with Lemma \ref{iso1} and Proposition \ref{iso}, we get that $M$ is a tilting module over $A$ and that $\End_A(M)\simeq\Gamma$. Therefore by Theorem 1.6 of \cite{Hap2} we have a derived equivalence $$\xymatrix{R\Hom_A(M,-):\Dd^b(A)\ar[r]^(.7)\sim & \Dd^b(\Gamma)}.$$ We still have to prove that the diagram
$$\xymatrix{\Dd^b(A)\ar[d]_{\pi_A}\ar[rr]^{R\Hom_A(M,-)}\ar[rd]_{F_A} &  & \Dd^b(\Gamma)\ar[ld]^{F_\Gamma}\\ 
 \Cc_A\ar[r]& \underline{\Sub}\Lambda_w &}$$ commutes.
The tilting $A$-module $M$ is sent to $\Gamma$ by the functor $R\Hom_A(M,-)$, and thus to the cluster-tilting object $\underline{C}_\ww$ in $\underline{\Sub} \Lambda_w$ by Theorem \ref{art}.
By definition of the generalized cluster category, the objects $\pi_A(\SSS (X))$ and $\pi_A(X)$ are isomorphic in the category $\Cc_A$, therefore we have an isomorphism in $\Cc_A$
$$\pi_A(M)\simeq \pi_A(A).$$  Hence by Theorem \ref{airt2} the object $M$ is sent to the cluster-tilting object $\underline{C}_\ww$ in $\underline{\Sub} \Lambda_w$. The triangle functors $\Dd^b(A)\rightarrow \underline{\Sub}\Lambda_w$ and $\Dd^b(\Gamma)\rightarrow \underline{\Sub}\Lambda_w$ are given by tensor products (see \cite{Ami3} and \cite{ART}). We can now conclude using the fact that two triangle functors which are tensor products and which coincide on a tilting object are isomorphic. Therefore the diagram above is commutative. And we finish the proof of Theorem \ref{derivedeq}.
 \end{proof}
 
\section{Example}

Let $Q$ be the quiver $\xymatrix@-1.3pc{ & 2\ar[dr] & \\ 1\ar[ur]\ar[rr] &
  & 3 }$ and $\ww:=s_3s_2s_3s_1s_2s_3s_1s_2s_3$. The word $\ww$  is co-$c$-sortable with $c=s_1s_2s_3$ and we have $c^{(0)}=c^{(1)}=s_1s_2s_3$, $c^{(2)}=s_2s_3$, $c^{(3)}=s_3$, $\ww'=s_3s_2s_3s_1s_2s_3$ and 
the quiver $Q^{(1)}$ is $Q$. It satisfies the orientation conditions $(*)$ and $(**)$.
The endomorphism algebra in $\underline{\Sub}\Lambda_w$ of the standard cluster-tilting object $\underline{C}_\ww$ of Theorem\ref{birsc} is the Jacobian algebra $\Jac(Q'_\ww,W_\ww)$ (Theorem \ref{birs}) where 
$$\xymatrix{&&2\ar@{<-}[rrr]^p\ar[dr]^c\ar@(d,l)[ddrr]_b &&& 5\ar[dr]^h & \\ Q'_\ww:=& 1\ar@{<-}[rr]_(.4)q\ar[ur]^a && 3\ar@{<-}[rrr]^(.7)r\ar[urr]^e\ar[dr]^d &&&
  6\\ &&&& 4\ar[uur]^(.6)f\ar[urr]_g &&}$$
and $W_\ww:= gdr + her+fbp-ecp+caq.$ The arrows $\{c,f,g,h\}$ are the $Q$-arrows, $\{a,b,d,e\}$ are the $Q^*$-arrows, and $\{p,q,r\}$ are the arrows going to the left.

Then $T$ is the module $$T= {\bsm 3 &&&\\&2&&3\\&&1&\esm}\oplus
{\bsm &&&&&&3\\&&&3&&2&\\3&&2&&1&&\\ &1&&&&&\esm}\oplus {\bsm
  3\\1\esm}=I_1\oplus\tau_{kQ^{(1)}}(I_2)\oplus T_3$$
The torsion class $\Mm=\Fac (T)$ has finitely many indecomposables, namely we have
$$\Mm=\{{\bsm 3\esm}, {\bsm 3\\2\esm}, {\bsm 3 &&&\\&2&&3\\&&1&\esm}, {\bsm &&&&3\\&3&&2& \\2&&1&&\esm },
{\bsm &&&&&&3\\&&&3&&2&\\3&&2&&1&&\\ &1&&&&&\esm}, {\bsm
  3\\1\esm}\}=\{I_1,I_2,I_3,\tau_{kQ}(I_3),\tau_{kQ}(I_2),T_3\}$$

The Auslander-Reiten quiver of $\Mm$ is
$$
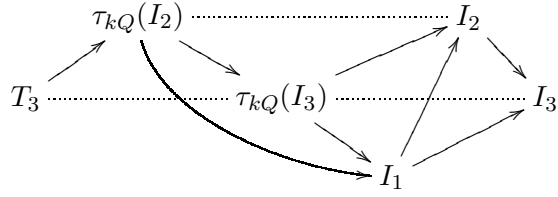
\xymatrix@-.4cm{&\tau_{kQ}(I_2)\ar@{..}[rrr]\ar[dr]\ar@(d,l)[ddrr] &&& I_2\ar[dr] & \\ T_3\ar@{..}[rr]\ar[ur] && \tau_{kQ}(I_3)\ar@{..}[rrr]\ar[urr]\ar[dr] &&&
  I_3\\ &&& I_1\ar[uur]\ar[urr] &&}$$
Therefore the algebra $A$ is given by the quiver
$$
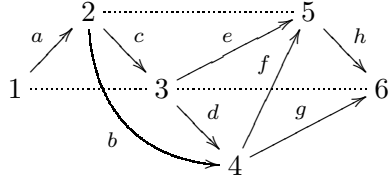
\xymatrix@-.3cm{&2\ar@{..}[rrr]\ar[dr]^c\ar@(d,l)[ddrr]_b &&& 5\ar[dr]^h & \\ 1\ar@{..}[rr]\ar[ur]^a && 3\ar@{..}[rrr]\ar[urr]^e\ar[dr]^d &&&
  6\\ &&& 4\ar[uur]^(.6)f\ar[urr]^g &&}$$
with the relations $he-gd=0$, $fb-ec=0$ and $ca=0$.

$$\SSS^{-2}(P_1)={\bsm 6 \esm},\quad  \SSS^{-1}(P_2)={\bsm 5\\3\esm}\quad \textrm{and}\quad  \SSS^{-1}(P_3)={\bsm 6\\5\esm}.$$
  
We easily check that the $A$-module $M$ of Theorem \ref{derivedeq} is $$M:= \SSS^{-2}(P_1)\oplus \SSS^{-1}(P_2\oplus P_3)\oplus (P_4\oplus P_5\oplus P_6)\simeq {\bsm 6 \esm}\oplus{\bsm 5\\3\esm}\oplus{\bsm 6\\5\esm}\oplus {\bsm
  &&4&&\\&3&&2&\\2&&&&1\esm}\oplus {\bsm & 5
  &&&\\3&&4&&\\&2&&3&\\&&&&2\esm}\oplus {\bsm
  &&&&6&&&\\&&&5&&4&&\\&&4&&3&&2&\\&3&&2&&&&1\\2&&&&&&&\esm}.$$
The endomorphism algebra $\End_A(M)$ is given by the quiver
 $$
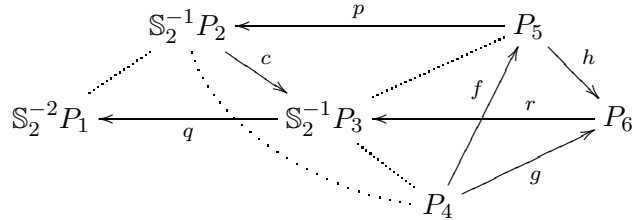
\xymatrix@-.3cm{&\SSS^{-1}P_2\ar@{<-}[rrr]^p\ar[dr]^c\ar@{..}@(d,l)[ddrr] &&&  P_5\ar[dr]^h & \\ \SSS^{-2}P_1\ar@{<-}[rr]_q\ar@{..}[ur] 
&&\SSS^{-1}P_3\ar@{<-}[rrr]^(.7)r\ar@{..}[urr]\ar@{..}[dr] &&&
 P_6\\ &&&P_4 \ar[uur]^(.6)f\ar[urr]_g &&}$$
with relations $rh-cp=0$, $rg=0$, $qc=0$ and $pf=0$. It is isomorphic to the algebra $\Gamma_\ww$ defined in section 2.1.


\begin{thebibliography}{20}

\bibitem[Ami09]{Ami3}
C.~Amiot, \emph{Cluster categories for algebras of global dimension 2 and
quivers with potential}, Ann. Inst. Fourier (2009), Vol. \textbf{59} no 6, pp 2525--2590.

\bibitem[AIRT11]{AIRT}
C.~Amiot, O.~Iyama, I.~Reiten, and G.~Todorov, \emph{Preprojective algebras and
  c-sortable words}, Proc. of the Lond. Math. Soc. (2011).

\bibitem[AO10]{AO}
C.~Amiot and S.~Oppermann, \emph{Cluster equivalences and graded derived
  equivalences}, preprint (2010), arXiv:1003.4916.


\bibitem[ART11]{ART}
C.~Amiot, I. Reiten, and G.~Todorov, \emph{The ubiquity of the generalized
cluster categories}, Adv. Math., \textbf{226} (2011), pp. 3813--3849.

\bibitem[AS81]{AS}
M. Auslander and S. O. Smalo, \emph{Almost split sequences in subcategories}, J. of Algebra, \textbf{69} (1981), 426--454.

\bibitem[BIRS09]{Bua2}
A.~B. Buan, O.~Iyama, I.~Reiten, and J.~Scott, \emph{Cluster structures for
  2-{C}alabi-{Y}au categories and unipotent groups}, Compos. Math. \textbf{145}
  (2009), 1035--1079.

\bibitem[BIRS11]{BIRSm}
A.~B. Buan, O.~Iyama, I.~Reiten, and D.~Smith, \emph{Mutation of
  cluster-tilting objects and potentials}, Amer. J. Math. \textbf{133} (2011), no. 4, 835--887.
\bibitem[BMR+06]{Bua}
A.~B. Buan, R.~Marsh, M.~Reineke, I.~Reiten, and G.~Todorov, \emph{Tilting
  theory and cluster combinatorics}, Adv. Math. \textbf{204} (2006), no.~2,
  572--618.

\bibitem[DWZ08]{DWZ}
H.~Derksen, J.~Weyman, and A.~Zelevinsky, \emph{Quivers with potentials and
  their representations. {I}. {M}utations}, Selecta Math. (N.S.) \textbf{14}
  (2008), no.~1, 59--119.
  
\bibitem[EMM10]{EMM}  
K. Erdmann, D. Madsen and V. Miemietz, \emph{On Auslander-Reiten translates in functorially finite subcategories and applications},  Colloq. Math. \textbf{119} (2010), 51--77. 

\bibitem[FZ02]{FZ1}
S.~Fomin and A.~Zelevinsky, \emph{Cluster algebras. {I}. {F}oundations}, J.
  Amer. Math. Soc. \textbf{15} (2002), no.~2, 497--529 (electronic).

\bibitem[GLS06]{Gei3}
C.~Geiss, B.~Leclerc, and J.~Schr{\"o}er, \emph{Rigid modules over
  preprojective algebras}, Invent. Math. \textbf{165} (2006), no.~3, 589--632.

\bibitem[GLS07a]{Gei}
C.~Geiss, B.~Leclerc, and J.~Schr{\"o}er, \emph{Auslander algebras and initial seeds for cluster algebras}, J.
  London Math. Soc. (2) \textbf{75} (2007), no.~3, 718--740.

\bibitem[GLS07b]{Gei4}
C.~Geiss, B.~Leclerc, and J.~Schr{\"o}er, \emph{Cluster algebra structures and semi-canonical bases for
  unipotent groups}, preprint (2007), arXiv:math. RT/0703039.

\bibitem[GLS08]{Gei2}
C.~Geiss, B.~Leclerc, and J.~Schr{\"o}er, \emph{Partial flag varieties and preprojective algebras}, Ann. Inst.
  Fourier (Grenoble) \textbf{58} (2008), no.~3, 825--876.

\bibitem[Hap87]{Hap2}
D.~Happel, \emph{On the derived category of a finite-dimensional algebra},
  Comment. Math. Helv. \textbf{62} (1987), no.~3, 339--389.

\bibitem[IO09]{IO}
O.~Iyama and S.~Oppermann, \emph{$n$-representation-finite algebras and
  $n$-{APR}-tilting}, Trans. Amer. Math. Soc. \textbf{363} (2011), no. 12, 6575--6614.
  
\bibitem[Kel05]{Kel}
B.~Keller, \emph{On triangulated orbit categories}, Doc. Math. \textbf{10}
  (2005), 551--581 (electronic).

\bibitem[KR08]{Kel4}
B.~Keller and I.~Reiten, \emph{Acyclic {C}alabi-{Y}au categories}, Compos.
  Math. \textbf{144} (2008), no.~5, 1332--1348, With an appendix by Michel Van
  den Bergh.

\bibitem[MRZ03]{Mar}
R.~Marsh, M.~Reineke, and A.~Zelevinsky, \emph{Generalized associahedra via
  quiver representations}, Trans. Amer. Math. Soc. \textbf{355} (2003), no.~10,
  4171--4186 (electronic).

\bibitem[Rea07]{Rea}
Nathan Reading, \emph{Clusters, {C}oxeter-sortable elements and noncrossing
  partitions}, Trans. Amer. Math. Soc. \textbf{359} (2007), no.~12, 5931--5958.

\end{thebibliography}
\end{document}